\documentclass[preprint,8pt]{elsarticle}
\usepackage{amssymb}
\usepackage{lineno,hyperref}
\usepackage{amsmath}
\usepackage{amsthm}
\usepackage{graphicx}
\usepackage{mathrsfs}
\newtheorem{lemma}{Lemma}[section]
\newtheorem{theorem}{Theorem}[section]
\newtheorem{remark}{Remark}[section]

\journal{}

\begin{document}

\begin{frontmatter}



\title{Cauchy problem for general time fractional diffusion equation}


\author[Firstaddress]{Chung-Sik Sin \corref{mycorrespondingauthor}}
\cortext[mycorrespondingauthor]{Corresponding author}
\ead{cs.sin@ryongnamsan.edu.kp}

\address[Firstaddress]{Faculty of Mathematics, \textit{\textbf {Kim Il Sung}} University, Ryomyong Street, Pyongyang, Democratic People's Republic of Korea}

\begin{abstract}
In the present work, we consider the Cauchy problem for the time fractional diffusion equation involving the general Caputo-type differential operator proposed by Kochubei.
First, the existence, the positivity and the long time behavior of solutions of the equation without source term are established by using the Fourier analysis technique.
Then, based on the representation of the solution of the inhomogenous  linear ordinary differential equation with the general operator, the similar problems for the diffusion equation with source term are studied.
\end{abstract}


\end{frontmatter} 


\section{Introduction}
\label{Section 1}
In this paper, we study the general time fractional diffusion equation of the form
\begin{align}
\label{governing_equation}
\mathscr{D}_{(k)} u(t,x)& =\triangle u(t,x)+h(t,x), t>0, x \in \mathbb{R}^n, \\
\label{initial_condition}
&u(0,x)=u_0(x), x \in \mathbb{R}^n.
\end{align}
where  $n\in \{1,2, 3\}$ and  $ \mathscr{D}_{(k)} $  denotes the general Caputo-type fractional differential operator defined by \cite{Kochubei}
\begin{equation}
\nonumber
\mathscr{D}_{(k)} v(t)=\frac{d}{dt}\int_0^t k(t-\tau)v(\tau)d\tau-k(t)v(0).
\end{equation}
Here $ k \in L_1^{loc}(\mathbb{R}_{+})$. 
The operator  $ \mathscr{D}_{(k)} $ stands for the Caputo fractional differential operator when $ k(t)=t^{-\alpha} $ for some $ 0<\alpha<1 $. Moreover, it generalizes  multi-term and distributed order fractional differential operator.\\
The symbol $ \triangle $ means the Laplacian defined by
\begin{equation}
\nonumber
\triangle w(x)=-F^{-1}(|\xi|^2Fw(\xi))(x), w\in H^2(\mathbb{R}^n), x\in \mathbb{R}^n,
\end{equation}
where $ F, F^{-1} $ are respectively Fourier transform and inverse Fourier transform defined by
\begin{align}
\nonumber
&Fw(\xi)=\tilde{w}(\xi)=\frac{1}{(2\pi)^\frac{n}{2}}\int_{\mathbb{R}^n} w(x) e^{-ix\xi} dx,\\
\nonumber
&F^{-1}w(x)=\frac{1}{(2\pi)^\frac{n}{2}} \int_{\mathbb{R}^n} w(\xi) e^{ix\xi} d\xi.
\end{align}
Here $H^2(\mathbb{R}^n)$ denotes the Sobolev space endowed with the norm defined by \cite[Chapter 4]{Demengel}
\begin{equation}
\nonumber
\|w\|_{H^2(\mathbb{R}^n)}=\|(1+|\cdot|^2)\tilde{w}\|_{L^2(\mathbb{R}^n)}.
\end{equation}
The regularities of the source term $ h $ and the initial value $ u_0 $ will be discussed later.

The equation (\ref{governing_equation}) is derived from the continuous time random walk  theory which is one of the powerful tools for describing the anomalous diffusion process of complex materials \cite{Sandev_FCAA}. The anomalous diffusion phenomena have attracted many researchers' attention \cite{Golding_Cox,Weber,Tejedor,Kneller,Stachura}. 
Due to the continuous time random walk  theory, $ u(t,x) $ means the probability density function  to find a walker at position $ x $ at time $ t $.
If the waiting time probability density function of the walker follows the power law asymptotically, we obtain the Caputo time fractional diffusion equation, which is a special type of the equation (\ref{governing_equation}). 
In \cite{Chechkin-Sokolov}, the ultraslow diffusion phenomena which can not be described by employing the usual fractional derivatives were successfully captured by using the distributed order fractional derivative including single-term and multi-term fractional derivatives. 
Since the equation (\ref{governing_equation}) is a generalization of the diffusion equation involving the distributed order fractional derivative, it has the ability to model more complex anomalous diffusion processes. 

The mathematical aspects for the anomalous diffusion equation have been extensively investigated.
In \cite{Mainardi_1996}, the fundamental solutions of Cauchy problem and Signalling problem for the Caputo time fractional diffusion equation without source term were expressed in terms of special functions by using the Laplace transform.
Eidelman and Kochubei investigated the fundamental solution of the Cauchy problem for the diffusion equation with the Caputo fractional derivative in time and a uniformly elliptic operator with variable coefficients in space \cite{Eidelman}.
Kochubei considered the fundamental solution and the classical solution of the ultraslow diffusion equation with the distributed order derivative \cite{Kochubei_2008}.
In \cite{Kochubei}, Kochubei proposed the general Caputo-type differential operator $ \mathscr{D}_{(k)} $ and investigated the LT-solution of the homogenous diffusion equation involving the general operator. In order that a function becomes a  LT-solution, its Laplace transform should be twice continuously differentiable with respect to $x$.  The kernel function of the operator $ \mathscr{D}_{(k)} $ considered in \cite{Kochubei}  satisfies the following conditions.\\
(C1) The Laplace transform $ \hat k $ of $ k $,
\begin{equation}
\nonumber
Lk(s)=\hat k(s)=\int_0^{\infty} k(t) e^{-ts} dt,
\end{equation}
exists for all $ s>0 $, \\
(C2) $ \hat k(s) $ is a Stieltjes function,\\
(C3) $ \hat k(s) \rightarrow 0 $ and $ s \hat k(s) \rightarrow \infty $ as $ s \rightarrow \infty $,\\
(C4) $ \hat k(s) \rightarrow \infty $ and $ s \hat k(s) \rightarrow 0 $ as $ s \rightarrow 0 $.\\
In \cite{Luchko_2010,Jiang,Jin,Li,Luchko-Yamamoto,Liu,Luchko_Yamamoto_2017},  by using the eigenfunction expansion of the symmetric uniformly elliptic operator and the properties of Mittag-Leffler type functions, the initial boundary value problems for single-term, multi-term and distributed order time fractional diffusion equations were studied. In particular, in \cite{Luchko-Yamamoto}, the authors pointed out that the asymptotic behavior of solutions of the Cauchy problem for the general fractional diffusion equation  (\ref{governing_equation}) remains open.

In this work, existence and analytical properties of solutions of the Cauchy problem are investigated. 
In particular, we find a result for the asymptotic behavior of solutions related to the open problem presented in \cite{Luchko-Yamamoto}.  
In the present paper,  we consider the solution which is Fourier transformable with respect to the spatial variable. 
In fact, when the equation (\ref{governing_equation}) is obtained from continuous time random walk  theory, it is supposed that $ u(t,x) $ is not only Laplace transformable with respect to  $ t $ but also Fourier transformable with respect to $ x $.
From the physical view point, it is not neccessary for the solution to be twice continuously differentiable with respect to $ x $.

The organization of the paper is as follows. In Section 2, the general fractional diffusion equation without source term is studied. The existence, the nonnegativity and the long time behavior of the solution are proved. In Section 3, we disuss the equation with source term subjected to zero initial condition.

\section{General time fractional diffusion equation without source term}
\label{Section 2}
\setcounter{section}{2}
\setcounter{equation}{0}\setcounter{theorem}{0}
In this section, the general diffusion equation (\ref{governing_equation}) with $ h=0 $ is investigated.
Throughout this paper, we suppose that the kernel function $k$ satisfies the conditions (C1)-(C4).
\begin{theorem}
	\label{Th.2.1}
Let $h=0$ and $ u_0 \in H^2(\mathbb{R}^n) $.
Then the Cauchy problem (\ref{governing_equation})-(\ref{initial_condition}) has a unique solution
$u\in  C([0,\infty); H^2(\mathbb{R}^n) \cap  L^\infty (\mathbb{R}^n))  $.
The following relations  hold.
\begin{align}
\label{p_L2}
&\|u(t,\cdot)\|_{L^2(\mathbb{R}^n)}\leq \|u_0\|_{L^2(\mathbb{R}^n)}, t \geq 0.\\
\label{p_Mk}
&\|u(t,\cdot)\|_{H^2(\mathbb{R}^n)} \leq \|u_0\|_{H^2(\mathbb{R}^n)}, t \geq 0.\\
\label{p_zero_Mk}
&\lim_{t \to 0}\|u(t,\cdot)-u_0\|_{H^2(\mathbb{R}^n)}=0.\\
\label{p_derivative_L2}
&\bigg\|\frac{\partial^j u(t,\cdot)}{\partial t^j}\bigg\|_{L^2(\mathbb{R}^n)}\leq \bigg(\frac{j}{et}\bigg)^j\|u_0\|_{L^2(\mathbb{R}^n)}, j \in \mathbb{N}, t>0 .
\end{align}
\begin{align}
\label{p_derivative_Mk}
&\bigg\|\frac{\partial^j u(t,\cdot)}{\partial t^j}\bigg\|_{H^2(\mathbb{R}^n)}\leq \bigg(\frac{j}{et}\bigg)^j\|u_0\|_{H^2(\mathbb{R}^n)},j \in \mathbb{N}, t>0 .\\
\label{Dg}
&\|D_{(k)}u(t,\cdot)\|_{L^2(\mathbb{R}^n)}\leq\|u_0\|_{H^2(\mathbb{R}^n)}, t>0.\\
\label{p_infty}
& \|{u(t,\cdot)}\|_{L^\infty (\mathbb{R}^n)}  \leq  K_n \|u_0\|_{L^2(\mathbb{R}^n)}   \bigg(\frac{\|k\|_{L^1(0,t)}}{t}\bigg)^{\frac{n}{4}}, t>0,
\end{align}
where $ K_n $ is a constant depending only on the dimension $n$.
If $  u_0  $ is nonnegative, then the solution is also nonnegative. 
In particular, if $  u_0  $ is bounded, then the solution is also bounded. 
\end{theorem}
\begin{proof}
Applying the Fourier transform to (\ref{governing_equation}) and (\ref{initial_condition}) with respect to the space variable $ x $, we have
\begin{equation}
\label{general_fractional_Laplace_diffusion_Fourier}
\mathscr{D}_{(k)} \tilde{u}(t,\xi) = -|\xi|^2 \tilde{u}(t,\xi), \ \  t>0, \  \xi \in \mathbb{R}^n
\end{equation}
\vskip -3pt \noindent
and
\vskip -13pt
\begin{equation}
\label{initial_condition_Fourier}
\tilde{u}(0,\xi)=\tilde{u}_0(\xi).
\end{equation}
It follows from  Theorem 3.2 in \cite{Sin_2018} that the initial value problem (\ref{general_fractional_Laplace_diffusion_Fourier})-(\ref{initial_condition_Fourier}) has a unique solution in the space of continuous functions.
Let $Y(t,\xi)$ be the solution of the equation (\ref{general_fractional_Laplace_diffusion_Fourier}) with the initial condition $\tilde{u}(0,\xi)=1$.
Then the solution of the problem (\ref{general_fractional_Laplace_diffusion_Fourier})-(\ref{initial_condition_Fourier}) has the  form: $\tilde{u}(t,\xi)=\tilde{u}_0(\xi)Y(t,\xi)$.
By Theorem 2 in \cite{Kochubei}, for $ \xi \in \mathbb{R}^n$, $Y(t,\xi)$ is completely monotone with respect to $ t $. Then we have
\vskip -11pt
\begin{equation}
\nonumber
|\tilde{u}(t,\xi)| \leq |\tilde{u}_0(\xi)|, t>0, \xi \in \mathbb{R}^n.
\end{equation}
It follows from Theorem 5.2 in \cite{Sin_2018} that $ Y(t,\xi)$ continuously depends on $\xi$.
Thus, $ Y(t,\xi)$ is measurable with respect to $ \xi $.  Since $ u_0\in L^2(\mathbb{R}^n) $,  $  \tilde{u}_0 \in L^2(\mathbb{R}^n) $. Then the function $ u(t,\cdot) $ is also in $ L^2(\mathbb{R}^n)$. Moreover, if $ u_0 \in H^2(\mathbb{R}^n) $, then $ u(t,\cdot) $ is also in $ H^2(\mathbb{R}^n) $.
Using the inverse Fourier transform, $ u(t,x) $ is obtained from $ \tilde{u}(t,\xi) $.

\vskip 3pt 

By using the Plancherel theorem, for $ t\geq 0 $, we deduce
\vskip -12pt
\begin{align}
\nonumber
&\|u(t,\cdot)\|_{L^2(\mathbb{R}^n)}\!=\!\|\tilde{u}(t,\cdot)\|_{L^2(\mathbb{R}^n)}\!=\!
\|\tilde{u}_0Y(t,\cdot)\|_{L^2(\mathbb{R}^n)}\!\leq \!\|\tilde{u}_0\|_{L^2(\mathbb{R}^n)}
\!=\!\|u_0\|_{L^2(\mathbb{R}^n)},\\
\nonumber
&\|u(t,\cdot)\|_{ H^2(\mathbb{R}^n)}\!=\!\|(1+|\cdot|^2)\tilde{u}(t,\cdot)\|_{L^2(\mathbb{R}^n)}
\!\leq \!\|(1+|\cdot|^2)\tilde{u}_0\|_{L^2(\mathbb{R}^n)} \!=\! \|u_0\|_{ H^2(\mathbb{R}^n)}.\end{align}
For $ t>0 $,  we estimate
\vskip -10pt
\begin{align}
\nonumber
&\|u(t,\cdot)-u_0\|_{H^2(\mathbb{R}^n)}=\|(1+|\cdot|^2)(\tilde{u}(t,\cdot)-\tilde{u}_0)\|_{L^2(\mathbb{R}^n)}\\
\nonumber
&=\|(1+|\cdot|^2)\tilde{u}_0(1-Y(t,\cdot))\|_{L^2(\mathbb{R}^n)}\leq 2\|u_0\|_{ H^2(\mathbb{R}^n)}.
\end{align}
By the Lebesgue dominated convergence theorem, we obtain
\begin{align}
\nonumber
&\lim_{t \to 0}\|u(t,\cdot)-u_0\|_{H^2(\mathbb{R}^n)}=\lim_{t \to 0} \|(1+|\cdot|^2)(\tilde{u}(t,\cdot)-\tilde{u}_0)\|_{L^2(\mathbb{R}^n)}\\
\nonumber
&=\|(1+|\cdot|^2)\tilde{u}_0\lim_{t \to 0} (1-Y(t,\cdot))\|_{L^2(\mathbb{R}^n)}=0.
\end{align}

By  the L'H\^{o}pital Rule and the property (C4), we can easily see that
\vskip -8pt
\begin{equation}
\label{infinity_property_of_g}
\lim_{t \to \infty} \frac{\|k\|_{L^1(0,t)}}{t}=0.
\end{equation}
Using the Young inequality for convolution, for $ \xi \in \mathbb{R}^n, t>0 $, we have
\begin{align}
\nonumber
&|\xi|^2 t Y(t,\xi) \leq \||\xi|^2 Y(\cdot,\xi) \|_{L^1(0,t)}=\|\mathscr{D}_{(k)} Y(\cdot,\xi) \|_{L^1(0,t)}\\
\nonumber
&\leq \|k\|_{L^1(0,t)}\int_0^t \bigg| \frac{\partial Y(s,\xi)}{\partial s}\bigg| ds= \|k\|_{L^1(0,t)}(1- Y(t,\xi)).
\end{align}
\vskip -3pt \noindent
Then  we obtain
\vskip -13pt
\begin{align}
\label{Xi_Y_inequaliy}
&|\xi|^2  Y(t,\xi) \leq \frac{\|k\|_{L^1(0,t)}}{t}, \ t>0 , \ \xi \in \mathbb{R}^n,\\
\label{Y_inequaliy}
&Y(t,\xi)\leq \frac{1}{1+\frac{t|\xi|^2}{\|k\|_{L^1(0,t)}}}, \ t>0 , \ \xi \in \mathbb{R}^n.
\end{align}
\vskip -4pt \noindent
For $ t>0 $, we have
\vskip -12pt
\begin{equation}
\nonumber
\|Y(t,\cdot)\|_{L^1(\mathbb{R}^n)}\leq \int_{\mathbb{R}^n} \frac{1}{1+\frac{t |\xi|^2}{\|k\|_{L^1(0,t)}}}d\xi.
\end{equation}
\vskip -4pt \noindent
If $ n=1$, then, for $ t>0 $, we deduce
\vskip -11pt
\begin{equation}
\nonumber
\|Y(t,\cdot)\|_{L^1(\mathbb{R})}\leq 2 \int_0^\infty  \frac{1}{1+\frac{t \xi^2}{\|k\|_{L^1(0,t)}}}d\xi=\bigg(\frac{\|k\|_{L^1(0,t)}}{t}\bigg)^{\frac{1}{2}}\pi.
\end{equation}
If $ n\geq1 $ and $ 2p> n $, then,   for $ t>0 $, we estimate
\begin{align}
\nonumber
\|Y(t,\cdot)\|^p_{L^p(\mathbb{R}^n)}&\leq \int_{\mathbb{R}^n} \frac{1}{\bigg(1+\frac{t |\xi|^2}{\|k\|_{L^1(0,t)}}\bigg)^p}d\xi= C_n \int_0^\infty  \frac{r^{n-1}}{\bigg(1+\frac{t r^2}{\|k\|_{L^1(0,t)}}\bigg)^p}dr\\
\nonumber
&\leq C_n \int_0^{\frac{\|k\|^{\frac{1}{2}}_{L^1(0,t)}}{t^{\frac{1}{2}}}}  r^{n-1}dr
+C_n \int_{\frac{\|k\|^{\frac{1}{2}}_{L^1(0,t)}}{t^{\frac{1}{2}}}}^\infty  \frac{r^{n-1}}{\bigg(\frac{t r^{2}}{\|k\|_{L^1(0,t)}}\bigg)^p}dr\\
\nonumber
&\leq \frac{C_n}{n} \bigg(\frac{\|k\|_{L^1(0,t)}}{t}\bigg)^{\frac{n}{2}}+\frac{C_n}{2p-n}\bigg(\frac{\|k\|_{L^1(0,t)}}{t}\bigg)^{\frac{n}{2}},
\end{align}
where $ C_n $ is a constant depending only on $n$.
Then, since $n\in \{1,2, 3\}$,  there exists a constant $ K_n $ depending only on $n$ such that  for $ t>0 $,
\vskip -10pt
\begin{equation}
\label{estimation_of_fundamental_solution}
\|Y(t,\cdot)\|_{L^2(\mathbb{R}^n)}\leq K_n \bigg(\frac{\|k\|_{L^1(0,t)}}{t}\bigg)^{\frac{n}{4}}.
\end{equation}

Now we set
\vskip -13pt
\begin{equation}
\nonumber
A(t,x)=F^{-1}(Y(t,\cdot))(x), \ t>0, \   x \in \mathbb{R}^n.
\end{equation}
Since $  Y(t,\cdot)\in L^2(\mathbb{R}^n) $, $ A(t,\cdot)\in L^2(\mathbb{R}^n) $  for $ t>0$.
Thus
\vskip -10pt
\begin{equation}
\nonumber
u(t,x)=\int_{\mathbb{R}^n}u_0(x-y) A(t,y) dy, \ t>0, \  x\in \mathbb{R}^n.
\end{equation}
Using the Young inequality for convolution and (\ref{estimation_of_fundamental_solution}), for $ t>0 $, we obtain
\begin{align}
\nonumber
& \|{u(t,\cdot)}\|_{L^\infty (\mathbb{R}^n)}  \leq \|u_0\|_{L^2(\mathbb{R}^n)}\|A(t,\cdot)\|_{L^2(\mathbb{R}^n)}
= \|u_0\|_{L^2(\mathbb{R}^n)} \| Y(t,\cdot)\|_{L^2(\mathbb{R}^n)}\\
\nonumber
& \leq K_n \|u_0\|_{L^2(\mathbb{R}^n)}   \bigg(\frac{\|k\|_{L^1(0,t)}}{t}\bigg)^{\frac{n}{4}}.
\end{align}
By using the Bernstein theorem, %[90p]%
\cite{Pruss}, p.90,  and the fact:
\vskip -10pt
\begin{equation}
\nonumber
\sup_{s\in [0,\infty)}s^je^{-ts}=\bigg(\frac{j}{et}\bigg)^j, \  j \in \mathbb{N},
\end{equation}
we can easily see that
\vskip -10pt
\begin{equation}
\nonumber
\bigg|\frac{\partial^j Y(t,\xi)}{\partial t^j}\bigg|\leq \bigg(\frac{j}{et}\bigg)^j, \ j \in \mathbb{N}, \ t>0,
\  \xi \in \mathbb{R}^n.
\end{equation}
Then, for $j \in \mathbb{N}, t>0 $,  we deduce
\begin{align}
\nonumber
&\bigg\|\frac{\partial^j \tilde{u}(t,\cdot)}{\partial t^j}\bigg\|_{L^2(\mathbb{R}^n)}=\bigg\|\tilde{u}_0\frac{\partial^j Y(t,\cdot)}{\partial t^j}\bigg\|_{L^2(\mathbb{R}^n)}\leq \bigg(\frac{j}{et}\bigg)^j\|\tilde{u}_0\|_{L^2(\mathbb{R}^n)}\\
\nonumber
&=\bigg(\frac{j}{et}\bigg)^j \|u_0\|_{L^2(\mathbb{R}^n)}.
\end{align}
Similarly, we can prove the relation (\ref{p_derivative_Mk}).
For $ t>0 $, we have
\begin{align}
\nonumber
&\|D_{(k)}u(t,\cdot)\|_{L^2(\mathbb{R}^n)}=\|D_{(k)}\tilde{u}(t,\cdot)\|_{L^2(\mathbb{R}^n)}=\||\cdot|^2\tilde{u}(t,\cdot)\|_{L^2(\mathbb{R}^n)}\\
\nonumber
&=\||\cdot|^2\tilde{u}_0Y(t,\cdot)\|_{L^2(\mathbb{R}^n)}\leq\||\cdot|^2\tilde{u}_0\|_{L^2(\mathbb{R}^n)}\leq \|u_0\|_{ H^2(\mathbb{R}^n)}.
\end{align}

It follows from Theorem 6.2  in \cite{Schilling} that the function $ s\hat{k}(s) $ is a complete Bernstein function.
Since the function $ e^{-s\hat{k}(s)} $ is a completely monotone function, by the Bernstein theorem %[90p]%
\cite{Pruss}, there exists a nondecreasing function $ \theta:[0, \infty) \rightarrow  \mathbb{R} $ such that
\vskip -10pt
\begin{equation}
\nonumber
L^{-1}\{\hat{k}(s)e^{-s\hat{k}(s)} \}(t)=\int_0^t k(t-\tau) d \theta(\tau).
\end{equation}
Define the function $ \psi(t,\tau) $ by
\begin{equation}
\nonumber
\psi(t,\tau)=L^{-1}\{\hat{k}(s)e^{-\tau s\hat{k}(s)} \}(t), t,\tau>0.
\end{equation}
We can easily see that for $ t,\tau>0,$  $ \psi(t,\tau)>0. $
By using the Fourier-Mellin inversion formula, for $ t>0 $, we deduce
\begin{equation}
\nonumber
\int_0^\infty \psi(t,\tau) d\tau= \int_0^\infty \frac{1}{2\pi i} \int_{r-i\infty}^{r+i\infty} \hat{k}(s)e^{st-\tau s\hat{k}(s)} ds d\tau
=\frac{1}{2\pi i} \int_{r-i\infty}^{r+i\infty} \frac{e^{st}}{s} ds=1,
\end{equation}
where $ r>0 $.
We have
\vskip -10pt
\begin{align}
\nonumber
&\int_0^\infty e^{-st} \int_0^\infty  \psi(t,\tau) e^{-|\xi|^2 \tau}d\tau dt
=\hat{k}(s)\int_0^\infty   e^{-\tau s\hat{k}(s)} e^{-|\xi|^2 \tau}d\tau \\
\nonumber
&=\frac{\hat{k}(s)}{s\hat{k}(s)+|\xi|^2}=\hat{ Y}(s,\xi), s>0, \xi \in \mathbb{R}^n.
\end{align}
\vskip -1pt \noindent
It follows from the uniqueness of the Laplace transform that $  Y(t,\xi) $ has the form:
\vskip -13pt
\begin{equation}
\nonumber
Y(t,\xi)=\int_0^{\infty}  \psi(t,\tau) e^{-|\xi|^2\tau}  d\tau, t>0, \xi \in \mathbb{R}^n.
\end{equation}
\vskip -1pt \noindent
By  Theorem 13.14 in \cite{Schilling}, $ |\xi|^2 $ is negative definite and by  Proposition 4.4 in \cite{Schilling}, for $ \tau>0$, $e^{-\tau |\xi|^2}$ is positive definite with respect to $ \xi $.
Then  $Y(t,\xi)$  is positive definite with respect to $ \xi $.
It follows from Bochner's theorem \cite[Theorem 4.14]{Schilling} that there exists a finite nonnegative measure $  \sigma_t $ on $ \mathbb{R}^n $ such that
\vskip -13pt
\begin{equation}
\nonumber
Y(t,\xi)=\int_{\mathbb{R}^n} e^{-i \xi x}\sigma_t (dx)=\tilde{\sigma_t}(\xi), \ t>0, \ \xi \in \mathbb{R}^n.
\end{equation}
\vskip -4pt \noindent
Then we have
\vskip -12pt
\begin{equation}
\nonumber
\sigma_t(\mathbb{R}^n)=\tilde{\sigma_t}(0)=Y(t,0)=\int_0^\infty \psi(t,\tau)d\tau=1, t>0.
\end{equation}
% \vskip -1pt \noindent
The solution of the Cauchy problem (\ref{governing_equation})-(\ref{initial_condition}) with $ h=0 $ has the form:
\vskip -8pt
\begin{equation}
\nonumber
u(t,x)=\int_{\mathbb{R}^n} u_0(x-y)1(y)\sigma_t(dy).
\end{equation}
% \vskip -3pt \noindent
Thus, if $ u_0\geq 0 $, then $ u \geq 0 $.
If $u_0$ is bounded, then
\vskip -10pt
\begin{equation}
\nonumber
|u(t,x)|\leq \int_{\mathbb{R}^n} |u_0(x-y)|1(y)\sigma_t(dy) \leq \sup_{x\in \mathbb{R}^n} |u_0(x)| \sigma_t(\mathbb{R}^n)=\sup_{x\in \mathbb{R}^n} |u_0(x)|.
\end{equation}

\end{proof}

We can prove the existence of solutions of the Cauchy problem (\ref{governing_equation})-(\ref{initial_condition}) under the initial condition weaker than Theorem \ref{Th.2.1}. %% 2.1. %% VK %% 

\vspace*{-3pt}

\begin{theorem} \label{Th.2.2} %% Th.2.2 %%% VK %% 
	Let $h=0$ and $ u_0 \in L^2(\mathbb{R}^n) $.  Then  the Cauchy problem (\ref{governing_equation})-(\ref{initial_condition}) has a unique solution
	$u\in   C([0,\infty); L^2(\mathbb{R}^n))$. \\ 
	Moreover, $u\in  \bigcap\limits_{\epsilon>0} C([\epsilon,\infty); H^2(\mathbb{R}^n)) $ and the following relations  hold$:$
	\begin{align}
	\label{p_L2_1}
	&\|u(t,\cdot)\|_{L^2(\mathbb{R}^n)}\leq \|u_0\|_{L^2(\mathbb{R}^n)}, \ t \geq 0,\\
	\label{p_Mk_1}
	&\|u(t,\cdot)\|_{H^2(\mathbb{R}^n)} \leq \bigg(1+\frac{\|k\|_{L^1(0,t)}}{t}\bigg) \|u_0\|_{L^2(\mathbb{R}^n)}, \ t > 0,\\
	\label{p_zero_Mk_1}
	&\lim_{t \to 0}\|u(t,\cdot)-u_0\|_{L^2(\mathbb{R}^n)}=0,\\
	\label{Dg_1}
	&\|D_{(k)}u(t,\cdot)\|_{L^2(\mathbb{R}^n)}\leq \frac{\|k\|_{L^1(0,t)}}{t} \|u_0\|_{L^2(\mathbb{R}^n)}, \ t>0,\\
	\label{p_infty_1}
	& \|{u(t,\cdot)}\|_{L^\infty (\mathbb{R}^n)}  \leq  K_n    \bigg(\frac{\|k\|_{L^1(0,t)}}{t}\bigg)^{\frac{n}{4}}\|u_0\|_{L^2(\mathbb{R}^n)}, \ t>0,
	\end{align}
	where $ K_n $ is a constant depending only on the dimension $n$.
	If $  u_0  $ is nonnegative, then the solution is also nonnegative.
	If $  u_0  $ is bounded, then the solution is also  bounded.
	If $ u_0 \in L^1(\mathbb{R}^n) \bigcap L^2(\mathbb{R}^n) $, then the solution $u$ of the Cauchy problem (\ref{governing_equation})-(\ref{initial_condition}) satisfies the following inequality
	\vskip -10pt
	\begin{equation}
	\label{p_L2}
	\|{u(t,\cdot)}\|_{L^2 (\mathbb{R}^n)}  \leq  K_n    \bigg(\frac{\|k\|_{L^1(0,t)}}{t}\bigg)^{\frac{n}{4}}\|u_0\|_{L^1(\mathbb{R}^n)}, \  t>0.
	\end{equation}
\end{theorem}

\begin{proof}
We can prove the result similarly to Theorem \ref{Th.2.1}. %%%% , %% 2.1, %%
% we can prove the result. % 
In particular, the relations (\ref{p_Mk_1}) and (\ref{Dg_1}) are obtained from the inequality (\ref{Xi_Y_inequaliy}) and the estimate (\ref{p_L2}) is proved by the Young inequality for convolution and (\ref{estimation_of_fundamental_solution}).
\end{proof} %%%%%%%%%%%

\begin{remark} % R.2.1 %
	When $ k(t)=\dfrac{t^{-\alpha}}{\Gamma(1-\alpha)} $, the $ L^2$-decay estimate (\ref{p_L2}) for solution of (\ref{governing_equation})-(\ref{initial_condition})  is
	\begin{equation}
	\label{p_L2_special}
	\|{u(t,\cdot)}\|_{L^2 (\mathbb{R}^n)}  \leq  \frac{K_n}{\Gamma(2-\alpha)}   t^{-\frac{n\alpha}{4}}\|u_0\|_{L^1(\mathbb{R}^n)}, \ \ t>0,
	\end{equation}
	\vskip -4pt \noindent
	which is the same as that from Theorem 4.2 in \cite{Kemppainen_2016}.  In order to use the result of \cite{Kemppainen_2016}, we have to find the function $ l \in L^1_{loc}(0,\infty) $ satisfying the relation
	\begin{equation}
	\nonumber
	\int_0^t k(t-s)l(s)ds=1,\ \ t>0.
	\end{equation}
	In general, it is not easy to find the analytic expression of the function $ l $.
	Meanwhile, the estimate (\ref{p_L2_special}) is also obtained from Theorem 2.22 in \cite{Kemppainen_2017}.
\end{remark} %%%%%%%%%%%%%

\section{General time fractional diffusion equation with source term}
\label{Section 3}
\setcounter{section}{3}
\setcounter{equation}{0}\setcounter{theorem}{0}
In this section, the general diffusion equation (\ref{governing_equation}) with the initial condition $ u(0,x)=0 $ is investigated.

\begin{lemma}
\label{representation_of_solution_0_intial_condition}
Let $ \lambda>0, T>0 $ and $ f \in C[0, T] $.
Then the fractional differential equation
\begin{equation}
\label{nonhomogenous_linear_differential_equation}
\mathscr{D}_{(k)}w(t)+\lambda w(t)=f(t), t>0
\end{equation}
with the initial condition $ w(0)=0 $
has a unique solution in $ C[0,T] $. In particular, the solution has the form
\begin{equation}
\label{nonhomogenous_solution_to_ODE}
w(t)=-\frac{1}{\lambda}\int_0^t v'(t-s)f(s)ds, 
\end{equation}
where $ v(t) $ is the solution of the following initial value problem
\begin{align}
\nonumber
&\mathscr{D}_{(k)}v(t)+\lambda v(t)=0, t>0, \\
\nonumber
&v(0)=1.
\end{align}
\end{lemma}
\begin{proof}
First, First, we prove that the function of the form (\ref{nonhomogenous_solution_to_ODE}) satisfies (\ref{nonhomogenous_linear_differential_equation}). For $ t>0 $, we have
\vskip -12pt
\begin{align}
\nonumber
\mathscr{D}_{(k)}w(t)&=\frac{d}{dt}\int_0^t k(t-s)w(s)ds\\
\nonumber
&=-\frac{1}{\lambda} \frac{d}{dt}\int_0^t \int_0^s k(t-s) v'(s-\tau)f(\tau)d\tau ds % \\
\nonumber
\end{align}  \begin{align}
&=-\frac{1}{\lambda} \frac{d}{dt}\int_0^t \int_\tau^t k(t-s) v'(s-\tau)dsf(\tau) d\tau\\
\nonumber
&=-\frac{1}{\lambda} \frac{d}{dt}\int_0^t \int_0^{t-\tau} k(t-\tau-r) v'(r)dr f(\tau)  d\tau\\
\nonumber
&=-\frac{1}{\lambda} \frac{d}{dt}\int_0^t  \mathscr{D}_{(k)}v(t-\tau) f(\tau)  d\tau\\
\nonumber
&=-\frac{1}{\lambda} \frac{d}{dt}\int_0^t  (-\lambda)v(t-\tau) f(\tau)  d\tau\\
\nonumber
&=\int_0^t v'(t-\tau) f(\tau)  d\tau+f(t)\\
\nonumber
&=-\lambda w(t)+f(t).
\end{align}
\vskip -4pt \noindent
For $ t>0 $, we estimate
\vskip -12pt
\begin{equation}
\nonumber
|w(t)|\leq -\frac{1}{\lambda} \max_{s\in [0,t]}|f(s)| \int_0^t v'(s)ds=\frac{1}{\lambda} \max_{s\in [0,t]}|f(s)|(1-v(t)),
\end{equation}
\vskip -3pt \noindent
which implies that $ \lim\limits_{t \to 0+} w(t)=0 $. Thus, $ w\in C[0,T] $.
The uniqueness of the solution follows from  Theorem 4.1 in \cite{Sin_2018}.
\end{proof}
In \cite{Sin_2016}, we derived a new property of the two-parameter Mittag-Leffler function to prove the existence results for solutions of the Caputo fractional differential equation.

\begin{remark}
	If  $ f \in L^\infty(0,T) $, then the function $w$ given by (\ref{nonhomogenous_solution_to_ODE}) satisfies the equation (\ref{nonhomogenous_linear_differential_equation}) for almost all $t\in(0,T) $.
\end{remark}

\begin{theorem}
\label{Th.3.1} %% VK %%%
Let $ u_0=0 $ and $  \dfrac{1}{ tk(t)} \in L^1_{loc}(0,\infty) $. Let $ h $ be a function satisfying $ h(t,\cdot)\in L^1(\mathbb{R}^n) $ for all $t \geq 0$ and suppose that there exists a function $ q\in H^2(\mathbb{R}^n)$ such that $ |\tilde{h}(t,\xi)| \leq q(\xi)  $ for $ t \geq 0, \xi \in \mathbb{R}^n $. Then the Cauchy problem (\ref{governing_equation})-(\ref{initial_condition}) has a unique solution
$u\in \bigcap\limits_{T>0} C([0,T]; H^2(\mathbb{R}^n) )$.
The following relations  hold$:$
\begin{align}
\label{u_L2}
&\|u(t,\cdot)\|_{L^2(\mathbb{R}^n)} \leq \|q\|_{L^2(\mathbb{R}^n)} \int_0^t  \frac{1}{ sk(s)} ds, \ t>0,\\
\label{u_H2}
&\|u(t,\cdot)\|_{H^2(\mathbb{R}^n)} \leq  \|q\|_{L^2(\mathbb{R}^n)}\bigg(1+ \int_0^t  \frac{1}{ sk(s)} ds\bigg), \ t>0,\\
\label{u_Dg}
&\|D_{(k)}u(t,\cdot)\|_{L^2(\mathbb{R}^n)}\leq 2\|q\|_{L^2(\mathbb{R}^n)}, \ t>0.
\end{align}
If $ h  $ is nonnegative, then the solution is also nonnegative.
\end{theorem}
\begin{proof}
Applying the Fourier transform to (\ref{governing_equation}) and (\ref{initial_condition}) with respect to the space variable $ x $, we have
\begin{equation}
\label{source_fractional_diffusion_Fourier}
\mathscr{D}_{(k)} \tilde{u}(t,\xi) = -|\xi|^2 \tilde{u}(t,\xi)+\tilde{h}(t,\xi), \ t>0, \ \xi \in \mathbb{R}^n
\end{equation}
\vskip -3pt \noindent
and
\vskip -13pt
\begin{equation}
\label{zero_initial_condition_Fourier}
\tilde{u}(0,\xi)=0.
\end{equation}
By Lemma \ref{representation_of_solution_0_intial_condition}, for $ \xi \in \mathbb{R}^n $, the solution of the initial value problem (\ref{source_fractional_diffusion_Fourier})-(\ref{zero_initial_condition_Fourier}) has  the  form
\vskip - 12pt
\begin{equation}
\label{with_source_solution_fourier}
\tilde{u}(t,\xi)=-\frac{1}{|\xi|^2}\int_0^t \tilde{h}(s,\xi) \frac{\partial Y(t-s,\xi)}{\partial s} ds, \ \ t>0, \  \xi \in \mathbb{R}^n.
\end{equation}
Here $Y(t,\xi)$ is the solution of the equation (\ref{general_fractional_Laplace_diffusion_Fourier}) with the initial condition $\tilde{u}(0,\xi)=1$.
Then we deduce
\vskip -10pt
\begin{equation}
\label{Z_derivative_estimation}
-tk(t)\frac{\partial Y(t,\xi)}{\partial t} \leq -\mathscr{D}_{(k)} Y(t,\xi) = |\xi|^2 Y(t,\xi).
\end{equation}
From (\ref{with_source_solution_fourier}) and (\ref{Z_derivative_estimation}),  for $t>0,  \xi \in \mathbb{R}^n$, we have
\begin{equation}
\nonumber
|\tilde{u}(t,\xi)|\leq \int_0^t |\tilde{h}(t-s,\xi)| \frac{Y(s,\xi)}{ sk(s)} ds \leq |q(\xi)| \int_0^t  \frac{1}{ sk(s)} ds.
\end{equation}
Then, for $  t>0 $, we have
\vskip -12pt
\begin{equation}
\nonumber
\|u(t,\cdot)\|_{L^2(\mathbb{R}^n)}=\|\tilde{u}(t,\cdot)\|_{L^2(\mathbb{R}^n)} \leq \|q\|_{L^2(\mathbb{R}^n)} \int_0^t  \frac{1}{ sk(s)} ds.
\end{equation}
Meanwhile, it follows from (\ref{with_source_solution_fourier}) that  for $t>0,\  \xi \in \mathbb{R}^n$,
\begin{equation}
\nonumber
|\xi|^2\tilde{u}(t,\xi)=-\int_0^t \tilde{h}(s,\xi) \frac{\partial Y(t-s,\xi)}{\partial s} ds\leq q(\xi).
\end{equation}
Then, for $t>0$, we estimate
\begin{equation}
\nonumber
\||\cdot|^2\tilde{u}(t,\cdot)\|_{L^2(\mathbb{R}^n)}\leq \|q\|_{L^2(\mathbb{R}^n)}.
\end{equation}
Thus we obtain (\ref{u_H2}). Also, using (\ref{source_fractional_diffusion_Fourier}),  for $ t>0 $,  we have
\begin{align}
\nonumber
&\|D_{(k)}u(t,\cdot)\|_{L^2(\mathbb{R}^n)}\!=\!\|D_{(k)}\tilde{u}(t,\cdot)\|_{L^2(\mathbb{R}^n)}\leq \||\cdot|^2\tilde{u}(t,\cdot)\|_{L^2(\mathbb{R}^n)}\!+\!\|\tilde{h}(t,\cdot)\|_{L^2(\mathbb{R}^n)}\\
\nonumber
&\leq 2\|q\|_{L^2(\mathbb{R}^n)}.
\end{align}
It follows from (\ref{with_source_solution_fourier}) and inverse Fourier transform that
\begin{equation}
\nonumber
u(t,x)=\int_0^t \int_{\mathbb{R}^n} h(s,x-y) B (t-s,y) dy ds, t>0,  x\in \mathbb{R}^n,
\end{equation}
where
\vskip -13pt
\begin{equation}
\nonumber
B(t,x)=F^{-1}\bigg(-\frac{1}{|\cdot|^2}\frac{\partial Y(t,\cdot)}{ \partial t}\bigg)(x).
\end{equation}
In fact, it follows from (\ref{Z_derivative_estimation}) that
\begin{equation}
\nonumber
-\frac{1}{|\xi|^2 }\frac{\partial Y(t,\xi)}{ \partial t} \leq \frac{ Y(t,\xi)}{tk(t)}, t>0,  \xi \in \mathbb{R}^n.
\end{equation}
Since $ Y(t,\cdot)\in L^2(\mathbb{R}^n) $ for $ t>0 $,  we have
\begin{equation}
\nonumber
-\frac{1}{|\cdot|^2}\frac{\partial Y(t,\cdot)}{ \partial t} \in L^2(\mathbb{R}^n), t>0.
\end{equation}
Thus, $ B(t, \cdot)\in L^2(\mathbb{R}^n) $ for $ t>0 $.
Moreover, by using (\ref{estimation_of_fundamental_solution}), for $ t>0 $, we obtain
\vskip -12pt
\begin{equation}
\nonumber
\|B(t,\cdot)\|_{L^2(\mathbb{R}^n)}\leq  \frac{\|Y(t,\cdot)\|_{L^2(\mathbb{R}^n)}}{tk(t)}  \leq K_n \frac{\|k\|^{\frac{n}{4}}_{L^1(0,t)}}{t^{1+\frac{n}{4}}k(t)}.
\end{equation}
It follows from the Young inequality for convolution that for $ t>0$,
\begin{equation}
\nonumber
\|u(t,\cdot)\|_{L^2(\mathbb{R}^n)} \!\leq \!\|q\|_{L^2(\mathbb{R}^n)}\! \int_0^t \! \frac{\|Y(s,\cdot)\|_{L^2(R^n)}}{sk(s)} ds \!\leq \! K_n \|q\|_{L^2(\mathbb{R}^n)} \! \int_0^t \! \frac{\|k\|^{\frac{n}{4}}_{L^1(0,s)}}{s^{1+\frac{n}{4}}k(s)} ds,
\end{equation}
which is an estimate sharper than (\ref{u_L2}).

As in the proof of Theorem \ref{Th.2.1}, there exists a nondecreasing function $ \mu:[0, \infty) \rightarrow  \mathbb{R} $ such that
\vskip -10pt
\begin{equation}
\nonumber
e^{-s\hat{k}(s)} =\int_0^\infty e^{-s\tau} d \mu(\tau).
\end{equation}
\vskip -3pt \noindent
Define $ \rho(t,\tau) $ by
\vskip -11pt
\begin{equation}
\nonumber
\rho(t,\tau)=L^{-1}\{e^{-\tau s\hat{k}(s)} \}(t), \ \ t,\tau>0.
\end{equation}
Meanwhile, we deduce
\begin{align}
\nonumber
&\int_0^\infty e^{-st} \int_0^\infty  \rho(t,\tau) e^{-|\xi|^2 \tau}d\tau dt
=\int_0^\infty   e^{-\tau s\hat{k}(s)} e^{-|\xi|^2 \tau}d\tau \\
\nonumber
&=\frac{1}{s\hat{k}(s)+|\xi|^2}=L \bigg(-\frac{1}{|\xi|^2}\frac{\partial Y(t,\xi)}{ \partial t}\bigg)(s),
\ \ s>0, \  \xi \in \mathbb{R}^n.
\end{align}
The uniqueness of the Laplace transform implies that
\begin{equation}
\nonumber
-\frac{1}{|\xi|^2}\frac{\partial Y(t,\xi)}{ \partial t}=\int_0^{\infty}  \rho(t,\tau) e^{-|\xi|^2\tau}  d\tau, \ \ 
t>0, \  \xi \in \mathbb{R}^n.
\end{equation}
Then, as in the proof of Theorem \ref{Th.2.1},  %% VK %% 
the function
\begin{equation}
\nonumber
-\frac{1}{|\xi|^2}\frac{\partial Y(t,\xi)}{ \partial t}
\end{equation}
is positive definite with respect to $ \xi $.
Then, by the Bochner theorem \cite[Theorem 4.14]{Schilling}, there exists a finite nonnegative measure
$  \Lambda_t $ on $ \mathbb{R}^n $ such that
\vskip -10pt
\begin{equation}
\nonumber
-\frac{1}{|\xi|^2}\frac{\partial Y(t,\xi)}{ \partial t}=\int_{\mathbb{R}^n} e^{-i \xi x}\Lambda_t (dx)=\tilde{\Lambda}_t(\xi), \ \ t>0, \ \xi \in \mathbb{R}^n.
\end{equation}
The solution of the Cauchy problem (\ref{governing_equation})-(\ref{initial_condition}) with $ u_0=0 $ has the form:
\begin{equation}
\nonumber
u(t,x)=\int_0^t \int_{\mathbb{R}^n} h(t-s,x-y) \Lambda_t(dy) ds, t>0,  x\in \mathbb{R}^n.
\end{equation}
Thus, if $ h\geq 0 $, then $ u \geq 0 $.
\end{proof}

We can prove the existence of solutions of the Cauchy problem (\ref{governing_equation})-(\ref{initial_condition}) under the conditions weaker than Theorem \ref{Th.3.1}. %% VK % 

\begin{theorem}\label{Th.3.2} %%% VK %% 
	Let $ u_0=0 $, $ h \in L^\infty(0,\infty; L^2(\mathbb{R}^n)) $ and suppose that there exists a function $ q\in L^2(\mathbb{R}^n)$ such that $ |\tilde{h}(t,\xi)| \leq q(\xi)  $ for almost all $ t \geq 0, \xi \in \mathbb{R}^n $. Then the Cauchy problem (\ref{governing_equation})-(\ref{initial_condition}) has a unique solution
	$u\in \bigcap\limits_{T>0} C([0,T]; H^2(\mathbb{R}^n))$.
	The following relations  hold$:$
	\begin{align}
	\label{u_L2_2}
	&\|u(t,\cdot)\|_{L^2(\mathbb{R}^n)} \leq t \|h\|_{L^\infty(0,\infty; L^2(\mathbb{R}^n))}, \ t>0,\\
	\label{u_H2_2}
	&\|u(t,\cdot)\|_{H^2(\mathbb{R}^n)} \leq  t \|h\|_{L^\infty(0,\infty; L^2(\mathbb{R}^n))}+\|q\|_{L^2(\mathbb{R}^n)},  \ t>0,\\
	\label{u_Dg_2}
	&\|D_{(k)}u(t,\cdot)\|_{L^2(\mathbb{R}^n)}\leq \|h\|_{L^\infty(0,\infty; L^2(\mathbb{R}^n))}+\|q\|_{L^2(\mathbb{R}^n)}, \text{ a.e. }  t>0.
	\end{align}
	If $ h  $ is nonnegative, then the solution is also nonnegative.
\end{theorem}

\begin{proof}
We can prove the result similarly to Theorem \ref{Th.3.1}. %% VK %%
\end{proof}

\end{document}